\documentclass{llncs}
\usepackage{amssymb,graphicx,color,enumerate}
\usepackage{amsmath,framed,cite}

\usepackage{graphicx}
\usepackage{algorithm}
\usepackage{algorithm}
\usepackage{float}
\usepackage{algpseudocode}

\title{Efficient algorithms for tuple domination on \\ co-biconvex graphs and web graphs \thanks{An extended abstract of this work appeared in the Proceedings of LAGOS 2019~\cite{DLLLAGOS2019}}}

\author{M.P. Dobson\inst{1}
\and
V. Leoni\inst{1,2}
\and
M. I. Lopez Pujato\inst{1,2}
\institute{FCEIA, Universidad Nacional de Rosario, Rosario, Santa Fe, Argentina\\
\email{pdobson@fceia.unr.edu.ar}
\and
CONICET, Argentina\\
\email{valeoni@fceia.unr.edu.ar, lpujato@fceia.unr.edu.ar}
}}
\begin{document}
\maketitle
				
\begin{abstract} 
A vertex in a graph dominates itself and each of its adjacent vertices. The \emph{$k$-tuple domination problem}, for a fixed positive integer $k$, is to find  a minimum sized vertex subset in a given graph such that every vertex is dominated by at least $k$ vertices of this set. 
From the computational point of view, this problem is NP-hard. It follows from previous works by Bui-Xuan et al.~(2013) and by Belmonte et al.~(2013)  ---in the context of locally checkable vertex subset problems in  graph classes with quickly computable and  bounded min-width--- that the $k$-tuple domination problem is  solvable in time $\mathcal{O}(|V(G)|^{6k+4})$ in the class of circular-arc graphs. 
In this work, we develop  faster algorithms for $k$-tuple domination in co-biconvex graphs and in web graphs, which are incomparable subclasses of concave-round graphs and thus of circular-arc graphs.
On the one hand, we  present an $\mathcal{O}(n^2)$-time algorithm for solving it for each $2\leq k\leq |U|+3$, where $U$ is the set of universal vertices and $n$ the total number of vertices of the input co-biconvex graph.
On the other hand, the study of this problem on web graphs was already started by Argiroffo et al. (2010)  from a polyhedral point of view only for the cases $k=2$ and $k=d(G)$,  where $d(G)$ equals the degree of  each vertex of the input web graph $G$. We complete this study for web graphs from an algorithmic point of view, by designing a linear-time algorithm based on the modular arithmetic for integer numbers. The algorithms presented in this work are mutually  independent but both exploit the circular properties of the augmented adjacency  matrices of each studied graph class.

\end{abstract}

\section{Introduction and preliminaries}\label{prem}

Domination in graphs is useful in various applications. In this work we address $k$-tuple domination.  Given a graph $G$ and a positive integer $k$, a  \emph{$k$-tuple dominating set} in $G$ is a set $D \subseteq V (G)$ satisfying that every vertex of $G$ has  in its closed neighborhood at least $k$ vertices of $D$. 1-tuple dominating sets are already well known as dominating sets. 

There exist many variations of domination. Among them, $k$-domination and total $k$-domination which regard  slight differences in their definitions. Specifically, given a graph $G$ and a positive integer $k$, a \emph{$k$-dominating set} in $G$ is a set $D\subseteq V(G)$ such that for every vertex $v\in V(G)\setminus D$, its open neighborhood has at least $k$ members from $D$, and
a \emph{total $k$-dominating set} in $G$ is a set $D\subseteq V(G)$ such that the open neighborhood of every vertex $v\in V(G)$ has at least $k$ members from $D$.

The differences in these three variations of domination make subclasses of tha class of  circular-arc graphs adequate and useful mostly due to their relation to ``circular'' issues, such as in forming sets of representatives, in resource allocation in distributed computing systems, in coding theory because of their relation to ``circular'' codes~\cite{Tucker}, and in testing for circular arrangements of genetic molecules~\cite{Klee}.

Concerning computational complexity results, the decision problem (fixed $k$) associated with these three variations of domination is NP-hard. In particular, $k$-tuple domination  is known to be  NP-hard  even for chordal graphs~\cite{LiaoChang2003}, and also hard to approximate~\cite{MR3027964}. As far as known, the class of strongly chordal graphs is the only maximal subclass of chordal graphs for which an efficient algorithm for every $k$ has been already developed~\cite{LiaoChang2003}. In another direction,  the problem is efficiently solved  for every $k$ on  $P_4$-tidy graphs~\cite{DLN2011} as well as on every graph class with bounded clique-width~\cite{MR3327090}.
As far as we  know, for interval graphs  and the case $k=1$, an efficient algorithm is developed in~\cite{chang}. Besides for the remaining values of $k$, an efficient algorithm is shown in~\cite{LiaoChang2003}  for strongly chordal graphs (which constitute a superclass of proper interval graphs). With a different approach,  polynomial-time algorithms were recently provided for $k$-domination and  total $k$-domination  for proper interval graphs, for each fixed value of $k$~\cite{LP1}.

For circular-arc graphs ---a superclass of interval graphs--- efficient algorithms are presented for the case $k=1$ in~\cite{chang} and~\cite{Hsu}, and for the case $k=2$ in recent work~\cite{Sinha}. 

Equivalently, a  $k$-tuple dominating set in a graph $G$ can be defined as a vertex subset $D$ such that for every vertex $v \in V (G)$, if it is in $D$ it has at least $k-1$ neighbors in $D$, and if it is not, it has at least $k$ neighbors in $D$. Following Definition 2 in \cite{Bui2013}, we  can then see   $k$-tuple domination problems belonging to the class of locally checkable vertex subset problems.   Since locally checkable vertex subset problems are solvable
in polynomial time in  graph classes with  bounded and quickly computable min-width~\cite{Bui2013} and circular-arc graphs have 
bounded (by 2) and quickly computable min-width~\cite{Belmonte2013},  we have  that $k$-tuple domination is  solvable in polynomial time. 
More specificaly,

\begin{sloppypar}
\begin{itemize}
  \item In~\cite{Bui2013}, Bui-Xuan et al.~give algorithms for a large class of locally checkable vertex problems.
     The algorithms run in time $\mathcal{O}(n^4\cdot {\it nec}_d(T,\delta)^3)$, where $n$ is the order of the input graph,
     $d$ is a parameter associated to the problem, and ${\it nec}_d(T,\delta)$ is the number of equivalence classes of a problem-specific equivalence relation on subsets of vertices, defined on a given tree decomposition  $(T,\delta)$ of $G$.
     For $k$-tuple domination problems, we have $d = k$, and hence the running time is $\mathcal{O}(n^4\cdot {\it nec}_k(T,\delta)^3)$.
  \item Belmonte and Vatshelle show in~\cite[Lemma 4]{Belmonte2013} that given an $n$-vertex circular-arc graph $G$ and a positive integer $d$, one can in polynomial time compute a tree decomposition  $(T,\delta)$ of $G$ with ${\it nec}_d(T,\delta) \le n^{2d}$.
\end{itemize}

Combining the above two results, we obtain, for every positive integer $k$, an algorithm for solving $k$-tuple domination in the class of $n$-vertex circular-arc graphs in time $\mathcal{O}(n^{6k + 4})$, as stated above.
\end{sloppypar}

\bigskip


We remark that none of the algorithms referenced above exploit the particular structure of the augmented adjacency matrices of each studied graph class. 

\subsection{Our results and approach}

In this work we address the $k$-tuple domination problem on two incomparable subclasses of the class of concave-round graphs  by exploiting the structure of the augmented adjacency matrices of these graphs. We develop novel and faster algorithms for these specific graph classes that avoid the exponential dependency of the value $k$.
In Section \ref{def}, we give some general definitions and notation.
In Section \ref{seccion3}, we explore  co-biconvex graphs and solve the problem for $k=2$ and $k=3$ for these graphs. This allows to develop an $\mathcal{O}(n^2)$-time algorithm for $k$-tuple domination for $2\leq k \leq |U|+3$, where $U$ is the set of universal vertices, if any, of the input graph with $n$ vertices. 
In Sections \ref{seccion4} and \ref{seccion5}, we generalize and complete the study of this problem on web graphs initiated ---by means of polyhedral arguments--- in~\cite{AEU}. Our approach for web graphs  is based on circular property of their augmented adjacency matrices and the modular arithmetic for integer numbers. We provide the exact value for the  minimum size of a $k$-tuple dominating set for every admissible $k$ and every web graph.

\section{Definitions and notation}\label{def}

In this work we consider finite simple graphs $G$, where $V(G)$ and $E(G)$ denote the vertex and edge sets, respectively. 
Given $S\subseteq  V(G)$, the subgraph of $G$ induced by $S$ (i.e., the one with vertex set $S$ and edge set $\{uv: uv\in E(G), \{u,v\}\subseteq  S\}$) is denoted by $G[S]$.
Given $S\subseteq  V(G)$, the induced subgraph $G[V(G)\setminus S]$ is denoted by $G-S$.
For simplicity, we write $G-v$ instead of $G-\{v\}$, for $v\in V(G)$. 

The \emph{(closed) neighborhood} of $v\in V(G)$ is $N_G[v] = N_G(v) \cup \{v\}$, where $N_G(v)=\{u\in V(G): uv \in E(G) \}$. The \emph{minimum degree} of $G$ is denoted by $\delta(G)$ and is the minimum between the cardinalities of $N_G(v)$ for all $v\in V(G)$.

A vertex $v \in V(G)$ is \emph{universal} in $G$ if $N_G[v]=V(G)$. 

A \emph{clique} in $G$ is a subset of pairwise adjacent vertices in $G$.

 A \emph{stable set} in $G$ is a subset of pairwise non-adjacent vertices in $G$ and the cardinality of a stable set of maximum cardinality in $G$ is denoted by $\alpha(G)$ and called the \emph{independence (or stability) number} of $G$.

A graph $G$ is an \emph{interval graph} if it has an intersection model consisting of intervals on the real line, that is, if there exists a family ${\cal I}$ of intervals on the real line and a one-to-one correspondence between the set of vertices of $G$ and the intervals of ${\cal I}$ such that two vertices are adjacent  in $G$ if and only if the corresponding intervals intersect. A \emph{proper interval graph} is an interval graph that has an intersection  model in which no interval contains another one. 

A graph $G$ is \emph{circular-arc} if  it has an intersection model consisting of arcs  on a circle, that is, if there is a one-to-one correspondence between the vertices of $G$ and a family of arcs on a circle such that two distinct vertices are adjacent in $G$ if and only if the corresponding arcs intersect. 

A \emph{$0,1$-matrix} is a matrix consisted of 0's (0-entries) and 1's (1-entries).
The square $0,1$-matrix whose entries are all 1's is denoted by $J$ and the identity matrix by $I$, both of appropriate sizes.

For a graph $G$, the \emph{adjacency matrix} of $G$ is the square  $0,1$-matrix $M(G)$ with $|V(G)|$ rows (and $|V(G)|$ columns) defined with entry $m_{ij}=1$ if and only if vertices $v_i$ and $v_j$ are adjacent. Note that $M(G)$ is symmetric and has 0's on the main diagonal. The  \emph{augmented adjacency matrix} or \emph{neighborhood matrix} $M^*(G)$ is defined as $M^*(G):=M(G) + I $, i.e.,  $M(G)$ with 1's added on the main diagonal.

A $0,1$-matrix has the \emph{consecutive 0's property} (C0P) for columns if there is a permutation of its rows that places the 0's consecutively in every column~\cite{Tucker}. Similarly, it has the \emph{consecutive 1's property} (C1P) for columns if there is a permutation of its rows that places the 1's consecutively in every column~\cite{fulkerson}. 

An ordering $<$ of $X$ in a bipartite graph $G$ with bipartition $X, Y$ of $V(G)$  has the \emph{adjacency property} if for every vertex $y\in Y$, $N_G(y)$ consists of vertices that are consecutive (an interval) in that ordering. 
A bipartite graph $G$  with bipartition $X, Y$ of $V(G)$ is \emph{biconvex} if there is an ordering of $X$ that fulfills the adjacency property.
\emph{Co-biconvex} graphs are the complementary graphs  of biconvex graphs~\cite{Safe2019}.

Notice that $M^*(G)$ has the C0P for columns if and only if $M(\bar{G})$ has the C1P for columns, where $\bar{G}$ stands for the complementary graph of $G$. Due to A. Tucker \cite{Tucker2}, $M(\bar{G})$ has the C1P for columns if and only if $G$  is co-biconvex:

A more general property is the following: a $0,1$-matrix has  the \emph{circular 1's property} (Circ1P) for columns if its rows can be permuted so that the  1's in each column are circular, i.e., appear in a circularly consecutive fashion by thinking of the matrix as wrapped around a cylinder  so that the first  row is adjacent to the last.  A. Tucker proved that all  graphs whose augmented adjacency matrices has the Circ1P for columns are circular-arc~\cite{Tucker}. Graphs whose augmented adjacency matrix has the Circ1P for columns are later called  \emph{concave-round} graphs in \cite{BanJ2000}. 
The resulta by  Tucker \cite{Tucker2} then can be stated as:, $M(\bar{G})$ has the C1P for columns if and only if $\bar{G}$  is bipartite and concave-round.

Combining results  by Tucker~\cite{Tucker} and Hell and Huang~\cite{Hell}, in~\cite{Safe2019} it is proved that co-biconvex  and more generally, all concave-round graphs are circular-arc graphs.

\bigskip

Given  $n, m \in \mathbb{Z}^{+}$ with $m\geq 1$ and $n\geq 2m+1$, a \emph{web} graph denoted by $W_n^m$  is a graph where $V(W_n^m)=\left\{v_1,\ldots,v_{n}\right\}$ and $v_iv_j\in E(W_n^m)$ if and only if $j\equiv i \pm l$ $\left(mod\;n\right)$, $l\in \left\{1,\ldots,m\right\}$~\cite{[7]}. It is clear that $\left|N_{W^m_n}[v]\right|=2m+1$ for $v\in V(W_n^m)$. 

Web graphs  are in particular examples of concave-round graphs. Examples of all graphs mentioned in this work and their corresponding augmented adjacency matrices are shown in Figures \ref{fig1}  and \ref{fig2} respectively.

\begin{figure}[ht]
\centering
\includegraphics[scale=1.1 ]{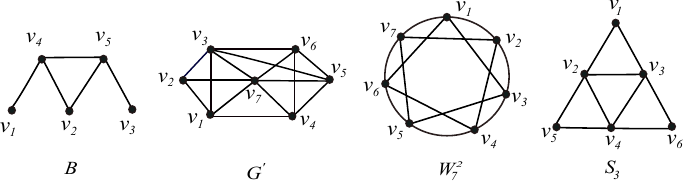}
 
\caption{ $B$ B is proper interval and neither co-biconvex nor a web graph, $G'$
is co-biconvex and neither
proper interval nor a web graph, the web graph  $W_7^2$
is neither co-biconvex nor proper interval,
and $S_3$ is concave-round but not co-biconvex, nor proper interval, nor a web graph.}\label{fig1}
\end{figure}

\begin{center}

\begin{figure}[ht]

$$\left(
\begin{array}{ccccc}
1&0&0&1&0\\
0&1&0&1&1\\
0&0&1&0&1\\
1&1&0&1&1\\
0&1&1&1&1
\end{array}
\right)
\; \;  \; \; \; \;   \left(
\begin{array}{ccccccc}
1 & 1 & 1 & 1 & 0 & 0 & 1\\
1 & 1 & 1 & 0 & 0 & 0 & 1\\
1 & 1 & 1 & 0 & 1 & 1 & 1\\
1 & 0 & 0 & 1 & 1 & 1 & 1\\
0 & 0 & 1 & 1 & 1 & 1 & 1\\
0 & 0 & 1 & 1 & 1 & 1 & 1\\
1 & 1 & 1 & 1 & 1 & 1 & 1\\
\end{array}
\right)
\; \; \; \; \; \; \left(
\begin{array}{ccccccc}
1&1&1&0&0&1&1\\
1&1&1&1&0&0&1\\
1&1&1&1&1&0&0\\
0&0&1&1&1&1&1\\
1&0&0&1&1&1&1\\
1&1&0&0&1&1&1\\
\end{array}
\right)
 \; \; \; \; \; \; \left(
\begin{array}{cccccc}
1&1&1&0&0&0\\
1&1&1&1&1&0\\
1&1&1&1&0&1\\
0&1&1&1&1&1\\
0&1&0&1&1&0\\
0&0&1&1&0&1\\
\end{array}
\right)$$

\caption{Augmented adjacency  matrices corresponding in the same order to graphs in Figure \ref{fig1}.}\label{fig2}
\end{figure}

\end{center}
%
%

Booth and  Lueker~\cite{booth2} devised an $\mathcal{O}(m+n+o)$-time algorithm to determine if a $0,1$-matrix of order $m \times n$ has the C1P for columns and to  obtain a desired row permutation when one exists, where $m$, $n$ and $o$ are respectively the number of rows, columns and 1's of the given matrix. Observe that a  $0,1$-matrix has the C0P for columns if and only if the matrix obtained by interchanging $0$'s by $1$'s and $1$'s by $0$'s has the C1P for columns. Since this exchange can be done in $\mathcal{O}(mn)$, for a square $0,1$-matrix we count on an $\mathcal{O}(n^2)$-time  algorithm to decide if it has the C0P for columns and to
 obtain a desired row permutation when one exists. 
Also, Booth and Lueker~\cite{booth2} devised a linear-time algorithm  for the Circ1P for columns.

%
%
%
%

Every numerical interval will be an integer interval, i.e.,  one of the form $\left[a,b\right]=\{x\in \mathbb{Z}: a \leq x\leq b\}$ with  endpoints $a,b\in \mathbb{Z}$ and $a<b$ together with all integer numbers  that are between $a$ and $b$.

The sets of integer numbers and positive integer numbers are denoted respectively by $\mathbb{Z}$ and $\mathbb{Z}^{+}$.

For a non-negative integer $k$, $D\subseteq V(G)$ is a \emph{$k$-tuple dominating set} in $G$ if $|N_G[v]\cap D|\geq k$, for every $v\in V(G)$. Notice that $G$ has $k$-tuple dominating sets if and only if $k \leq \delta(G)+1$ and, if  $G$ has a $k$-tuple dominating set $D$, then $|D|\geq k$. When $k\leq \delta(G)+1$, $\gamma_{\times k}(G)$ denotes the  cardinality of a  $k$-tuple dominating set in $G$ of minimum size and  $\gamma_{\times k}(G)=+\infty$, when $k> \delta(G)+1$. $\gamma_{\times k}(G)$ is called the \emph{$k$-tuple dominating number} of $G$~\cite{HararyH}. Observe that  $\gamma_{\times 1}(G)= \gamma(G)$, the usual domination number, i.e., the concept of tuple domination generalizes the well-known concept of domination in graphs.
Besides, note that $\gamma_{\times 0}(G)=0$  for every graph $G$. When $G$ is not connected, the $k$-tuple dominating number of $G$ is equal to the sum of the $k$-tuple dominating numbers of its connected components. Thus throughout this work, $G$ will be a  connected graph and the integer number $k$ will be less or equal $\delta(G) + 1$.

For a fixed positive integer $k$, the \emph{$k$-tuple domination problem} is to find in a given graph $G$, a $k$-tuple dominating set in $G$ of size $\gamma_{\times k}(G)$.


\section{$k$-tuple domination on co-biconvex graphs} \label{seccion3}

From the definition of tuple domination, it is clear that $\gamma_{\times k}(G)\geq k$ for every graph $G$ and positive integer $k$. 
Moreover, when $G$ has universal vertices, we can prove the following:

\begin{lemma}\label{Umasdek}
Let $G$ be a graph, $U$ the set of its  universal vertices and $k$ a positive integer. Then  $\gamma_{\times k}(G)=k$ if and only if $\left|U\right|\geq k$.
\end{lemma}

\begin{proof}
Suppose $\left|U\right|\geq k$ and take any subset of $k$ vertices from $U$. It is clear that this subset is a $k$-tuple dominating set in $G$ and thus  $\gamma_{\times k}(G) \leq k$. Then  $\gamma_{\times k}(G)=k$. 

Conversely, suppose $\gamma_{\times k}(G)=k$ and let $D$ be a minimum $k$-tuple dominating set in $G$. Since $|N[v] \cap D|\geq k$ for every $v\in V(G)$ and $|D|=k$, it turns out that $D \subseteq  N[v] $ for every $v\in V(G)$ impying that $D \subseteq U$. Therefore $|U|\geq |D|=k$.\end{proof}

Notice that, when $u$ is  a universal vertex in a graph $G$ and $D\subset V(G)$ is a $k$-tuple dominating set in $G$ with $u \notin D$, then by interchanging $u$ with any other vertex of $D$, we obtain a $k$-tuple dominating set of the same size containing $u$.
Then, the following relationship holds:


\begin{proposition}\label{masuno2}
Let  $G$ be a graph,  $U$ the set of its  universal vertices and $k$ a positive integer with $\left|U\right|\leq k-1$. Then   $$\gamma_{\times k}(G)=\gamma_{\times (k-\left|U\right|)}(G-U)+\left|U\right|.$$
\end{proposition}

\begin{proof}
Let $D$ be a $k$-tuple dominating set in $G$ with $\left|D\right|=\gamma_{\times k}(G)$. We know that $|D| \geq k$ and thus $|D| \geq |U|+ 1$.

From the observation above,   there exists  a  $k$-tuple dominating set $D'$ of $G$ such that $\left|D'\right|=\left|D\right|$ and $U\subseteq D'$. Then $D' \setminus U$ is a $(k-\left|U\right|)$-tuple dominating set in $G-U$ and thus  $\gamma_{\times(k-\left|U\right|)}(G-U) \leq\left|D'\right|-\left|U\right|=\gamma_{\times k}(G)-\left|U\right|. $

On the other hand, let $D$ be a minimum $(k-\left|U\right|)$-tuple dominating  set in $G-U$. It is clear that  $D\cup U$ is a $k$-tuple dominating set in $G$ since every vertex of $U$ is  universal in $G$.
Then $\gamma_{\times k}(G)\leq \left|D\cup U\right|=\left|D\right|+\left|U\right|=\gamma_{\times(k-\left|U\right|)}(G-U)+\left|U\right|$, and the proof is complete.\end{proof}

The following corollary is clear from Lemma \ref{Umasdek} and Proposition \ref{masuno2}.

\begin{corollary}
Let $G$ be  a graph, $U$  the set of its  universal vertices with $U\neq \emptyset$ and $k$ a positive integer.   If for some $ r\in \mathbb{Z}^+$ the value $\gamma_{\times i} (G-U)$ can be found in polynomial-time for $i=1, \ldots, r$,  then $\gamma_{\times k} (G)$ can be found in polynomial-time for every $k$ with $1 \leq k \leq |U|+r$.
\end{corollary}

\subsection{General properties of co-biconvex graphs}




We need to remark the following: 

\begin{remark} 
If $G$ is a co-biconvex graph then $G - U$ is a  co-biconvex graph, where $U$ is the set of universal vertices of $G$.
\end{remark}

\medskip

Let $G$ be a co-biconvex graph with  its vertices  indexed so that the $0$'s occur consecutively in each column of $M^*(G)$. We assume that  the columns of  $M^*(G)$  are ordered the same way as rows (which are ordered so that the C0P holds). Let  $C_1$ be the set of columns whose  0's are below  the main diagonal, $C_2$ the set of columns whose  0's are above  the main diagonal, and $C_3$ the set of columns without 0's. See graph $G'$ in Figure \ref{fig1} for an example, where $C_1$ corresponds to columns 1 to 3, $C_2$ to columns 4 to 6 and $C_3$ to column 7 of its corresponding augmented adjacency matrix in Figure \ref{fig2}. Sets $C_1$, $C_2$ and $C_3$ partition $V(G)$, $C_3$ corresponds to the set $U$ of universal vertices of $G$ and $C_1$ and $C_2$ are cliques in $G$: if a vertex $v_i \in C_1$ is non-adjacent to a vertex $v_j$, then  the corresponding 0 in column $j$ and row $i$ has to be above the diagonal since the 0 in column $i$ and row $j$ is below~\cite{Tucker}. We denote this partition by $(C_1, C_2, U)$.  When $U=\emptyset$, it follows that $|C_1|\geq 2$ and $|C_2|\geq 2$ (otherwise, w.l.o.g. if $|C_1|=|\{v_1\}|=1$, then $v_1$ must be adjacent to at least one vertex of $C_2$  since we are dealing wtih connected graphs; but then this vertex of $C_2$ is a universal vertex which is a contradiction. The same happens if $|C_1|=|\{v_1\}|=0$ since in this case $G$ is a complete graph and the $k$-tuple domination problem is already solved for complete graphs for every $k$. When $U=\emptyset$ we simply write $(C_1, C_2)$.  Also for simplicity, we denote $G_1:=G[C_1]$ and $G_2:=G[C_2]$.

From now on, $G$ is a co-biconvex graph and $(C_1, C_2, U)$ (or $(C_1, C_2)$ when $U =\emptyset$) is the above mentioned partition of $V(G)$.

Following the notation in~\cite{Tucker},  let us denote $V(G)=\{v_1,\ldots,v_n\}$, $C_1=\{v_1,\ldots,v_r\}$ and 
$C_2=\{v_{r+1},\ldots,v_n\}$ for a given co-biconvex graph $G$ with partition $(C_1, C_2)$.
Also let us denote by $M^*_{C_iC_j}$,  the submatrix of $M^{*}(G)$ with rows indexed by $C_i$ and columns by $C_j$. Notice that $M^*_{C_1C_1}$ and $M^*_{C_2C_2}$ are both equal to matrices $J$'s of appropriate sizes.  In this case, $M^*(G)$ looks like the scheme shown in Figure \ref{esquema}.

\begin{figure}[ht]
\begin{center}
\includegraphics[scale=1.2]{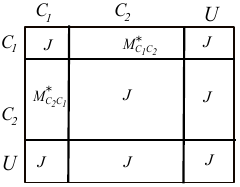}

\caption{Scheme of $M^*(G)$ for a co-biconvex graph $G$ with partition  $(C_1, C_2, U)$.}\label{esquema}
\end{center}
\end{figure}

\bigskip

It is easy to prove the following upper bound on the size of  a minimum $k$-tuple dominating set in a co-biconvex graph:

\begin{lemma} \label{contenidouniversal}
Let $G$ be a co-biconvex graph and $k$ a positive integer. If $|C_i|\geq k$  for $i=1,2$, then $$\gamma_{\times  k}(G)\leq 2k.$$
\end{lemma}

\begin{proof}
Let $D_i\subseteq C_i$  with  $|D_i|=k$, for $i=1,2$ and consider the set $D_1\cup D_2$. Take  $v\in V(G)$. If $v\in C_i$, then $D_i \subseteq  N_G[v]$, thus  $|N_G[v] \cap (D_1\cup D_2)| \geq |D_i| =k$, for $i=1$ or $i=2$. If $v\in U$, clearly $D_1 \cup D_2 \subseteq N_G[v]= V(G)$ and thus  $|N_G[v] \cap (D_1\cup D_2)| = |D_1 \cup D_2| =2k \geq k$. Then $D_1\cup D_2$ is a $k$-tuple dominating set in $G$ and the upper bound follows.\end{proof}

\bigskip

Proposition \ref{masuno2} allows us to restrict our study of  co-biconvex graphs to those  with partition $(C_1, C_2)$. In this case, it is clear that both $C_1$ and $C_2$ are non-empty sets, otherwise $G$ is complete and every vertex is universal. Since $G$ is connected, it follows that $|C_1|\geq 2$ and $|C_2|\geq 2$, since otherwise, the graph  would have universal vertices. Under these assumptions and Lemmas \ref{Umasdek} and \ref{contenidouniversal}, we have $k+1 \leq\gamma_{\times k}(G)\leq 2k$ for any co-biconvex graph $G$ without universal vertices.

\subsection{Construction of auxiliary interval graphs $H_i$}\label{aux}

Let $G$ be a co-biconvex graph without universal vertices and $(C_1, C_2)$  the above mentioned partition of $V(G)$.
We construct two interval graphs $H_1$ and $H_2$ as follows:
\begin{itemize}
\item for each vertex $v_i\in C_1$,  define an integer interval $I_i\subseteq \left[r+1,n\right]$ such that, if the  consecutive $0$'s of column $v_i$ correspond to the vertices $v_p,...,v_{p+s}$ where $p\geq r+1$ and  $p+s\leq n$, then $I_i=\left[p,p+s\right]$;

\item for each vertex $v_i \in C_2$,  define an integer interval $I_i \subseteq\left[1,r\right]$ such that, if the consecutive $0$'s of column $v_i$ correspond to the vertices $v_p,...,v_{p+s}$ with $p\geq 1$ and $p+s\leq r$, then $I_i=\left[p,p+s\right]$. 
\end{itemize}

We will consider that $v_i$ represents the interval $I_i$, for each $i=1,...,n$. 

The two interval graphs $H_1$ and $H_2$ defined  above have interval models $\mathcal{I}_1=\left\{I_1,I_2,...,I_r\right\}$ and $\mathcal{I}_2=\left\{I_{r+1},I_{r+2},...,I_n\right\}$, respectively. 

For a co-biconvex graph $G$ with partition $(C_1, C_2, U)$ and $U \neq \emptyset$, graphs $H_1$ and $H_2$ are defined as above from the subgraph $G-U$ of $G$.

An example of the above construction is shown in Figure \ref{auxiliary}:

\bigskip
\begin{figure}[ht]
\centering
\includegraphics[scale=1.2]{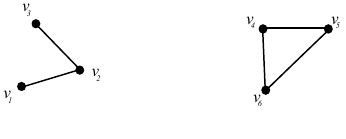}
\caption{On the left, graph $H_1$ and on the rigth, graph $H_2$, both related to graph $G'-v_7$ of Figure \ref{fig1}.}\label{auxiliary}
\end{figure}

It is clear that given two intersecting intervals $I_i$ and $I_j$ of $H_1$ for  $1\leq i \neq j \leq r$, there exists  $q$ with $r+1\leq q\leq n$ such that $m^*_{qi}=m^*_{qj}=0$. This means that $v_qv_i\notin E(G)$ and $v_qv_j\notin E(G)$. 
In other words, given two non-intersecting intervals $I_i$ and $I_j$ of $H_1$  for  $1\leq i \neq j \leq r$, we have $m^*_{qi}=1$ or $m^*_{qj}=1$  for all $q$ with $r+1\leq q\leq n$. Therefore in  each row $q$ of $M^*_{C_2C_1}$ there exists at least one  $1$-entry in the columns corresponding to vertex  $v_i$ or $v_j$, and then $v_qv_i\in E(G)$ or $v_q v_j\in E(G)$  for all $q$ with $r+1\leq q\leq n$. 

In a similar way, this argument clearly holds  for the interval graph $H_2$.

\bigskip
\subsection{Stable sets of $H_i$ and tuple-dominating sets of $G$}\label{stable}

We will denote by $\alpha_i$ the independence number of the interval graphs $H_i$ defined in the previous subsection, for $i \in \{1,2\}$. A linear-time algorithm to compute the independence number of an interval graph can be obtained for instance by a dynamic programming approach, traversing the vertices following an ordering $\{v_1, \dots, v_n\}$, of the vertices such that for all $1 \leq i < j < k \leq n$, the presence of the edge $v_iv_k$ implies the presence of the edge $v_jv_k$. Such an ordering can be  computed in linear-time \cite{olariu}.


\bigskip




When considering stable sets of $H_i$, the following interesting fact will be the key of the results in the next section: 

\begin{proposition}\label{estable} Let $G$ be a co-biconvex graph with partition $(C_1,C_2)$ and $S\subseteq {C_i}$, for some $i\in \{1, 2\}$. Then $S$ is a stable set of $H_i$ if and only if every vertex of $C_j$ has at least $|S|-1$ neighbors in $G$ belonging to $S$, for $i\neq j$.
\end{proposition}

\begin{proof}
That a vertex $v$ of $G_j$ has at least $|S|-1$ neighbors in $G$ belonging to  $S$ for $i\neq j$ means that $|N_{G}[v]\cap S|\geq |S|-1$. In other words,  for each row of $M^*_{C_jC_i}$ there exists at most one zero in  the columns  corresponding to  vertices in $S$, meaning that the elements of $\{I_t\}_{t:v_t\in S}$ are pairwise non-adjacent, i.e., $S$ is a stable set of $H_i$.
\end{proof}

\subsection{The algorithm for co-biconvex graphs}

The relationship exhibited in the previous section  between tuple dominating sets of a given co-biconvex graph and stable sets of the auxiliary interval graphs $H_1$ and $H_2$ defined from it allows us to state the following general result for the $k$-tuple dominating numbers of co-biconvex graphs.

\begin{theorem}\label{alfas}
Let $G$ be a co-biconvex graph with partition $(C_1, C_2)$, interval graphs $H_i$ defined as in the previous section, and $\alpha_i$ be the independence number of $H_i$, for each $i \in \{ 1, 2 \}$. Then

\begin{enumerate}
\item if  $1 \leq i\neq j \leq2$, $\alpha_i=1$ and $D$ is a  $k$-tuple dominating set in $G$, then  $|D\cap C_j|\geq k$;
\item if $\alpha_1= \alpha_2=1$ then $\gamma_{\times k}(G)=2k$;             
\item if $\alpha_1+ \alpha_2>k$ then $\gamma_{\times k}(G)=k+1$; 
\item if $\alpha_1+ \alpha_2= k$ and $|C_i|\geq \alpha_i+1$ for $i \in \{1,2\}$ then $\gamma_{\times k}(G)=k+2$. 
\end{enumerate}
\end{theorem}

\begin{proof}
\begin{enumerate}
\item  W.l.o.g., assume $i=1$. Then  $\alpha_1=1$ implies that the vertices in $H_1$ are pairwise adjacent. Hence, the corresponding intervals are pairwise overlapping. It is known that the interval model of an interval graph fulfills the Helly property (for the definition of this property, see for example \cite{HD64}). It follows that there is a point that is part of every interval. Hence, there is a row $j$ in $M^*_{C_2C_1}$ that contains only 0's and thus
vertex $v_j\in C_2$ is non-adjacent to every vertex in $C_1$. This implies that $|D\cap C_2|\geq k$ for each  $k$-tuple dominating set  $D$ of  $G$.

\item If $\alpha_1=\alpha_2=1$ and $D$ is a $k$-tuple dominating set in  $G$, then the previous item implies that $|D\cap C_j|\geq k$, for $j=1,2$ and thus it happens that $D$ has at least $2k$ vertices.  Thus  $\gamma_{\times  k}(G) \geq 2k$.  Since $|C_j|\geq |D\cap C_j|\geq k$, for each $j=1,2$, the desired equality follows from Lemma \ref{contenidouniversal}.

\item Let $S_1$ and $S_2$ be stable sets of $H_1$ and $H_2$ respectively, with $|S_1\cup S_2|=k+1$. Proposition \ref{estable} implies that  $S_i$ is a $|S_i|$-tuple dominating set in $G_i$ and also that every vertex of $C_j$  has at least $|S_i|-1$ neighbors in $G$ belonging tio $S_i$  for each $i,j \in \{ 1, 2 \}$ and  $i\neq j$. Thus  $S_1\cup S_2$ is a $k$-tuple dominating set in $G$ and then  $\gamma_{\times k}(G) \leq k+1$. Since $\gamma_{\times k}(G) > k$, from Lemma \ref{Umasdek} we conclude that $\gamma_{\times k}(G) = k+1$. 

\item  Let $S_1$ and $S_2$ be maximum stable sets of $H_1$ and $H_2$ respectively. It is clear that $S_1\cup S_2$ is a $(\alpha_1+\alpha_2-1)$-tuple dominating set in $G$, i.e., a $(k-1)$-tuple dominating set in $G$. Since $|C_i| \geq \alpha_i  + 1 \geq |S_i|+1$  for $i = 1, 2$, there exist two vertices $w_1\in C_1-S_1$ and $w_2\in C_2-S_2$. The set $S_1\cup S_2\cup \{w_1,w_2\}$ is a $k$-tuple dominating set in $G$ with cardinality $k+2$, implying $\gamma_{\times k}(G)\leq k+2$. 

Now, since $\gamma_{\times k}(G) \geq k+1$ ($U=\emptyset$), it is enough to show that $\gamma_{\times k}(G)\neq k+1$. Suppose  $D$ is a minimum $k$-tuple dominating set in $G$ with $|D|=k+1$ and denote $D_1=D\cap C_1$, $D_2=D\cap C_2$, $d_1=|D_1|$ and $d_2=|D_2|$. 
W.l.o.g. we assume $\alpha_1 < d_1$. It follows that $D_1$ is not a stable set of $H_1$ and thus, by Proposition \ref{estable}, there exists $v \in C_2$ such that $|N_G[v] \cap D_1|\leq d_1 - 2$. Therefore it holds  $|N_G[v] \cap D|\leq d_1 - 2 + d_2 = k - 1$, contradicting the fact that $D$ is a $k$-tuple dominating set in $G$.

\end{enumerate}
\end{proof}

The results up to now allow us to completely solve the problems for the cases $k=2$ and $k=3$, as shown in the following two subsections.

\subsubsection{$2$-tuple domination}

\begin{theorem}\label{2tuple}  Let $G$ be a co-biconvex graph with  partition $(C_1, C_2, U)$, interval graphs $H_i$ defined as in the previous section, and $\alpha_i$ the independence number of $H_i$ for each $i \in \{ 1, 2 \}$. 
\begin{enumerate}
\item If $|U|=1$ then $\gamma_{\times 2}(G)=  3$.
\item If $|U|\geq 2$ then $\gamma_{\times 2}(G)=2$.
\item If $|U|=0$ and $\alpha_1+\alpha_2\geq 3$  then $\gamma_{\times 2}(G)=3$.
\item If $|U|=0$ and  $\alpha_1=\alpha_2=1$ then $\gamma_{\times 2}(G)=4$.
\end{enumerate}
\end{theorem}

\begin{proof} 
\begin{enumerate}
\item From Lemma \ref{masuno2} we have $\gamma_{\times k}(G)=\gamma_{\times (k-1)}(G-U)+1$. Since $G-U$ has no universal vertex, Lemma \ref{Umasdek} implies  $\gamma_{\times 1}(G-U)\geq 2$. The set $\{w_1,w_2\}$ is a $1$-tuple dominating set in $G-U$, where $w_1\in C_1$ and $w_2\in C_2$ are any two vertices. Thus $\gamma_{\times 1}(G-U)= 2$. 
\item  Follows from Lemma \ref{Umasdek}.
\item  Follows from Theorem \ref{alfas} item iii.
\item  Follows from Theorem \ref{alfas} item ii.
\end{enumerate}
\end{proof}

\subsubsection{$3$-tuple domination}

\begin{theorem}\label{3tuple}  Let $G$ be a co-biconvex graph with  partition $(C_1, C_2, U)$, interval graphs $H_i$ defined as in the previous section, and $\alpha_i$ the independence number of $H_i$ for each $i \in \{ 1, 2 \}$. 

\begin{enumerate}

\item If $|U|=1$, then $\gamma_{\times 3}(G)= 4$ if $\alpha_1+\alpha_2\geq 3$, and $\gamma_{\times 3}(G)= 5$ if $\alpha_1+\alpha_2=2$.
\item If $|U|=2$ then $\gamma_{\times 3}(G)=4$.
\item If $|U|\geq 3$ then $\gamma_{\times 3}(G)= 3$.
\item If $|U|=0$ and $\alpha_1+\alpha_2\geq 4 $ then $\gamma_{\times 3}(G)=4$.
\item If $|U|=0$ and $\alpha_1=\alpha_2=1$ then $\gamma_{\times 3}(G)=6$.
\item If $|U|=0$ and $\alpha_1+ \alpha_2=3$ then $\gamma_{\times 3}(G)=5$.
\end{enumerate}
\end{theorem}

\begin{proof}   
Similarly as in the previous theorem, the proof follows by applying accordingly the  results in previous sections. In particular, to prove item (vi) we proceed in the following manner: suppose w.l.o.g. that $\alpha_1=1$ and $\alpha_2=2$. Then, since $G$ is connected,  it follows that $|C_1|\geq 2$ since otherwise $C_1=\{v\}$ for some $v\in V(G)$ and any vertex in $C_2$ adjacent to $v$ would be a universal vertex in $G$. Also, that  $|C_2|\geq 3$ since $|C_2| \geq |C_2 \cap D|\geq 3$ when $D$ is a $3$-tuple dominating set in $G$. Applying Theorem \ref{alfas} item (iv) we arrive at $\gamma_{\times 3}(G)=5$.
\end{proof}

\begin{corollary}
 The $k$-tuple domination problem can be solved in $\mathcal{O}(n^2)$-time on a co-biconvex graph $G$ with $n$ vertices for each $2\leq k\leq |U|+3$, where $U$ is the set of universal vertices of $G$.
\end{corollary}

\begin{proof} Given a co-biconvex graph $G$ and $k$ fixed, if $k \leq |U|$, then follow Lemma \ref{Umasdek}. Otherwise,  follow the next scheme:
\begin{enumerate}
\item [1.] Construct the augmented adjacency matrix $M^*(G)$.
\item [2.] Apply the $\mathcal{O}(n^2)$-time algorithm in \cite{booth2} to permute accordingly rows and columns of $M^*(G)$ to ensure the structure shown in Figure \ref{esquema}, where $n=|V(G)|$.
\item  [3.] Build the interval graphs $H_1$ and $H_2$ as explained  in  Section \ref{aux} and find the independence number of $H_i$ for $i=1,2$. As already pointed out, this can be done in linear-time \cite{olariu}.
\item [4.] Apply Theorem \ref{2tuple} and Theorem \ref{3tuple}.
Following Proposition \ref{masuno2} the proof is complete.
\end{enumerate}
\end{proof}

%
%
%
%

Next we apply the previous findings to the graph $G'$ of Figure \ref{fig1}.

\begin{example}
Recall the graph $G'$ from Figure \ref{fig1} and the auxiliary graphs $H_1$ and $H_2$ of Figure \ref{auxiliary}. The results exposed in this section can be applied appropriately in order to compute the values of $\gamma_{\times i}(G')$ for each $i \in \{ 1, 2, 3,4 \}$. Figure \ref{ejemplo} shows the graph $G'-U$. Actually,  since $\alpha_1=2$ and $\alpha_2=1$, we have :

$$\gamma_{\times 4}(G')=\gamma_{\times 3}(G' - v_7) +1= 5+1=6,$$

$$\gamma_{\times 3}(G')=\gamma_{\times 2}(G' - v_7) +1= 3+1=4,$$

$$\gamma_{\times 2}(G')=3$$ and

$$\gamma_{\times 1}(G')= 1.$$

\bigskip

\begin{figure}[ht]
\centering
\includegraphics[scale=1.1 ]{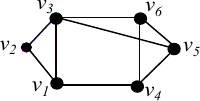}

\caption{Graph $G'- U$, where $G'$ is the graph of Figure \ref{fig1} and $U=\{v_7\}$.}\label{ejemplo}
\end{figure}
\end{example}

\section{Vertex set partition of web graphs}\label{seccion4}

Let us introduce the specific notation for the remainder of the paper.

Given  $a, b \in \mathbb{Z}$, the greatest common divisor between $a$ and $b$, i.e.,  the largest positive integer number that divides both $a$ and $b$,  is denoted by 
$gcd(a,b)$.

Given  $a, b \in \mathbb{Z}$ with $a\leq b$ and $m \in \mathbb{Z}^{+}$,  $a$ is \emph{congruent} with  $b$ modulo $m$ (denoted by $a\equiv b\;(mod\;m)$) if $b-a$ is a multiple of $m$.  For $i$ and $m$ in $\mathbb{Z}$, the set of all integer numbers that are congruent with $i$ (called the \emph{equivalence class} of $i$ modulo $m$) derived from this equivalence relation is denoted by  $\left[i\right]_m$. It is well-known that $\left\{\left[i\right]_m\right\}^{m-1}_{i=0}$ constitutes a partition of $\mathbb{Z}$, for each $m$.

\bigskip

From the definition of web graphs, it is clear that their augmented adjacency matrices satisfies the Circ1P for columns, thus web graphs are concave-round.

To simplify the notation, in the remainder  we omit the subscripts in every vertex neighborhood of a web $W_n^m$, and thus $N[v]$ will always mean $N_{W^m_n}[v]$.

\begin{center}
\begin{figure}[ht]
$\;\;\;\;\;\;\;\;\;\;\;\;\;\;\;\;\;\;\;\;\;\;\;\;$\includegraphics[scale=0.15]{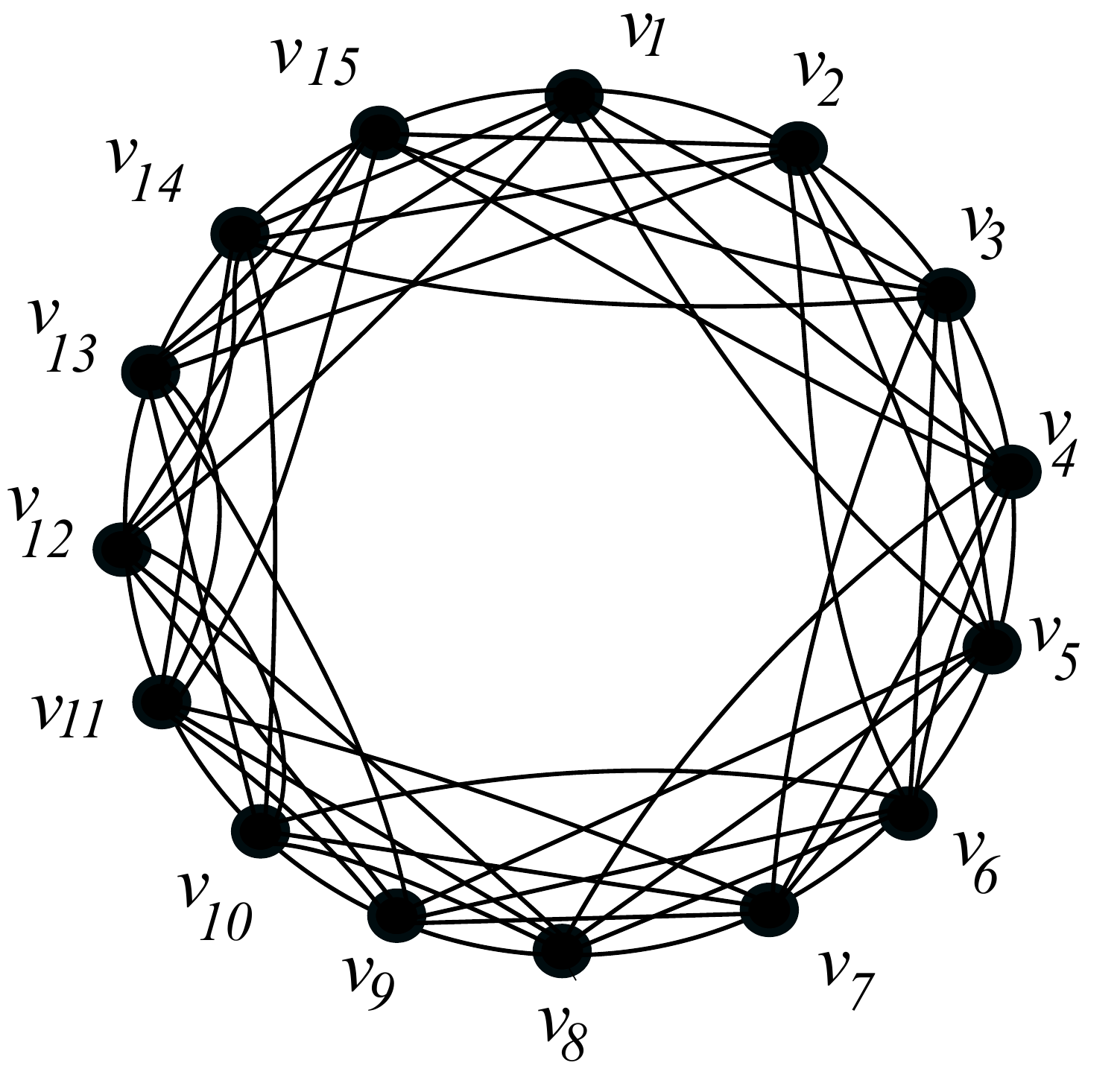}

$ \;\;\;\;\;M^*\left(W_{15}^4\right)= \left(\begin{array}{cccccccccccccccccccccccccccccc}
1&1&1&1&1&0&0&0&0&0&0&1&1&1&1\\
1&1&1&1&1&1&0&0&0&0&0&0&1&1&1\\
1&1&1&1&1&1&1&0&0&0&0&0&0&1&1\\
1&1&1&1&1&1&1&1&0&0&0&0&0&0&1\\
1&1&1&1&1&1&1&1&1&0&0&0&0&0&0\\
0&1&1&1&1&1&1&1&1&1&0&0&0&0&0\\
0&0&1&1&1&1&1&1&1&1&1&0&0&0&0\\
0&0&0&1&1&1&1&1&1&1&1&1&0&0&0\\
0&0&0&0&1&1&1&1&1&1&1&1&1&0&0\\
0&0&0&0&0&1&1&1&1&1&1&1&1&1&0\\
0&0&0&0&0&0&1&1&1&1&1&1&1&1&1\\
1&0&0&0&0&0&0&1&1&1&1&1&1&1&1\\
1&1&0&0&0&0&0&0&1&1&1&1&1&1&1\\
1&1&1&0&0&0&0&0&0&1&1&1&1&1&1\\
1&1&1&1&0&0&0&0&0&0&1&1&1&1&1\\

\end{array}\right)$
 
\caption{$W_{15}^4$ and its augmented adjacency matrix}\label{figure1}
\end{figure}

\end{center}
 
Given $n, m \in \mathbb{Z}^{+}$  with $n\geq 2m+1$, let us consider the integer division of $n$ by $2m+1$ and denote by $c$ the quotient and by $r$ the remainder,  with $0\leq r<2m+1$; i.e.,
$$n=c(2m+1)+r.$$
Also, let us denote  $\mu:=gcd\left(2m+1,r\right)$. Then, there  exist $l_1, l_2 \in \mathbb{Z}^{+}$ such that $r=l_1\mu$  and $2m+1=l_2\mu$, and then $n=\left(cl_2+l_1\right)\mu$.

\smallskip


The following lemmas are straightforward.

\begin{lemma}\label{coprimos}
$gcd\left(l_2, cl_2+l_1\right)=1$. 
\end{lemma}

\begin{lemma}\label{tamanocircular}
The set $\left[i\right]_\mu\cap \left[1,n\right]$ has cardinality $\frac{n}{\mu}$ for every $i \in \{1, \ldots,\mu\}$.
\end{lemma}

\begin{proof} It follows from the facts that $n$ is a multiple of $\mu$ and $\left\{\left[i\right]_\mu \cap \left[1,n\right]\right\}_{i=1}^{\mu}$ is a partition of $\left[1,n\right]$ into sets of the same cardinality.
\end{proof}

\smallskip   
In the sequel, we will denote for each $i\in\{1, \ldots,\mu\}$, 

\begin{equation}\label{eq1}
S_i:=\left[i\right]_\mu\cap \left[1,n\right].
\end{equation}

\begin{example}\label{ejemplo1} 
For $n=15$ and $m=4$ we have $c=1$, $2m+1=9$, $r=6$ and  $\mu=gcd(9,6)=3$. Sets $S_i$'s are:
$S_1=\left[1\right]_3\cap \left[1,15\right]=\left\{1,4,7,10,13\right\}$, $S_2=\left[2\right]_3\cap \left[1,15\right]=\left\{2,5,8,11,14\right\}$ and 
$S_3=\left[3\right]_3\cap \left[1,15\right]=\left\{3,6,9,12,15\right\}$.
\end{example}

The next expression for the sets $S_i$ will lead us to define a vertex set ordering in a web graph, which will be crucial in the following section.

\begin{proposition}\label{tamanocircular2}
For each $i\in\{1, \ldots,\mu\}$ it holds

$$S_i=\bigcup_{t \in \left[0,\; {n}/{\mu-1}\right]} \left\{w\in \left[1,n\right]: w \equiv i+t(2m+1)\;(mod\;n)\right\}.$$
\end{proposition}

\begin{proof} 

It is straightforward that for $i\in\{1, \ldots,\mu\}$,  $$\bigcup_{t \in [0,{n}/{\mu}-1]} \left\{w\in \left[1,n\right]: w \equiv i+t(2m+1)\;(mod\;n)\right\} \subseteq S_i.$$  From Lemma \ref{tamanocircular}, it is sufficient to prove that for each $i\in\{1, \ldots,\mu\}$, it holds
$$\left|\bigcup_{t \in \left[0,{n}/{\mu}-1\right]} \left\{w\in \left[1,n\right]: w \equiv i+t(2m+1)\;(mod\;n)\right\}\right|=\frac{n}{\mu}.$$

Let $i\in \{1, \ldots,\mu\}$ and $0\leq t_1<t_2\leq \frac{n}{\mu}-1$. If $i+t_1(2m+1)\equiv
i+t_2(2m+1)\;(mod \;n)$, then there exists $s\in\mathbb{Z}^{+}$ such that
$\left(t_2-t_1\right)(2m+1)=sn$.
Recall that $r=l_1\mu$ and  $2m+1=l_2\mu$, for $l_1, l_2\in \mathbb{Z}^+$, and then $n=\left(cl_2+l_1\right)\mu$.
Thus
$\left(t_2-t_1\right)l_2\mu=s\left(cl_2+l_1\right)\mu$, implying that 
$s\left(cl_2+l_1\right)$ is a multiple of $l_2$.
From Lemma \ref{coprimos}, it follows that  $s$ is a multiple of $l_2$, i.e., $s=\alpha l_2$
for some  $\alpha \in \mathbb{Z}^{+}$.
Therefore $t_2-t_1=\alpha \frac{n}{\mu}\geq \frac{n}{\mu} $, a contradiction.
\end{proof}

Let us  apply  the above result to the vertex set of a web graph.
From now on, all sums are taken modulo $n$.

Consider a web graph $W^m_n$, for some $m$ and $n$. From the definition of a web graph, it is clear that  the neighborhood of vertex $v_{m+j}$ is $$N[v_{m+j}]=\left\{v_j,v_{j+1},\ldots,v_{j+2m}\right\},$$ where all subscripts belong to $\left[1,n\right]$.

For simplicity and w.l.o.g., in the sequel we will refer to vertex $v_{m+j}$  as $j$, for each $j\in\left\{1,\ldots,n\right\}$.

In this way, as a corollary of Lemma \ref{tamanocircular} and Proposition \ref{tamanocircular2},  we can state:

\begin{corollary}\label{particion}
$\{S_i\}^\mu_{i=1}$ is a partition of  $V(W_n^m)$ into sets of cardinality $\frac{n}{\mu}$.
\end{corollary}

\section{ $k$-tuple domination on  web graphs}\label{seccion5}

In this section we derive an algorithm that computes $\gamma_{\times k}(W_n^{m})$ for a web graph $W_n^{m}$ and each $k$. 

Firstly, we recall a result in~\cite{AEU} where both an upper bound and a lower bound for   $\gamma_{\times k}(W_n^m)$  for every $k$ were given:  $$k\left\lfloor \frac{n}{2m+1}\right\rfloor\leq\gamma_{\times k}(W_n^{m})\leq k \left\lceil \frac{n}{2m+1}\right\rceil.$$ 

Then we note that the sets $S_i$'s defined in Section \ref{seccion4} play a crucial role when considering tuple dominating sets of a web graph. An example is the scheme shown in Figure \ref{figure3}.

\begin{lemma}\label{l_2domino}
Let $W^m_n$ be a web graph. Then for every vertex $v\in V(W^m_n)$ and  $i\in\left\{1,\ldots,
\mu\right\}$, $\left|N\left[v\right]\cap S_i\right|=l_2$, where $2m+1=l_2\mu$ for $l_2 \in \mathbb{Z}^+$.
\end{lemma}

\begin{proof} 
Let $i\in\left\{1,\ldots, \mu\right\}$.
From Lemma \ref{tamanocircular} and since $2m+1=l_2\mu$, in every interval of length $2m+1$, there are exactly $l_2$ elements of $S_i$, for each $i$. In particular this holds for $N[v]$ for every vertex $v$, and the result follows.\end{proof}

\begin{figure}[ht]
\centering
\includegraphics[scale=0.15]{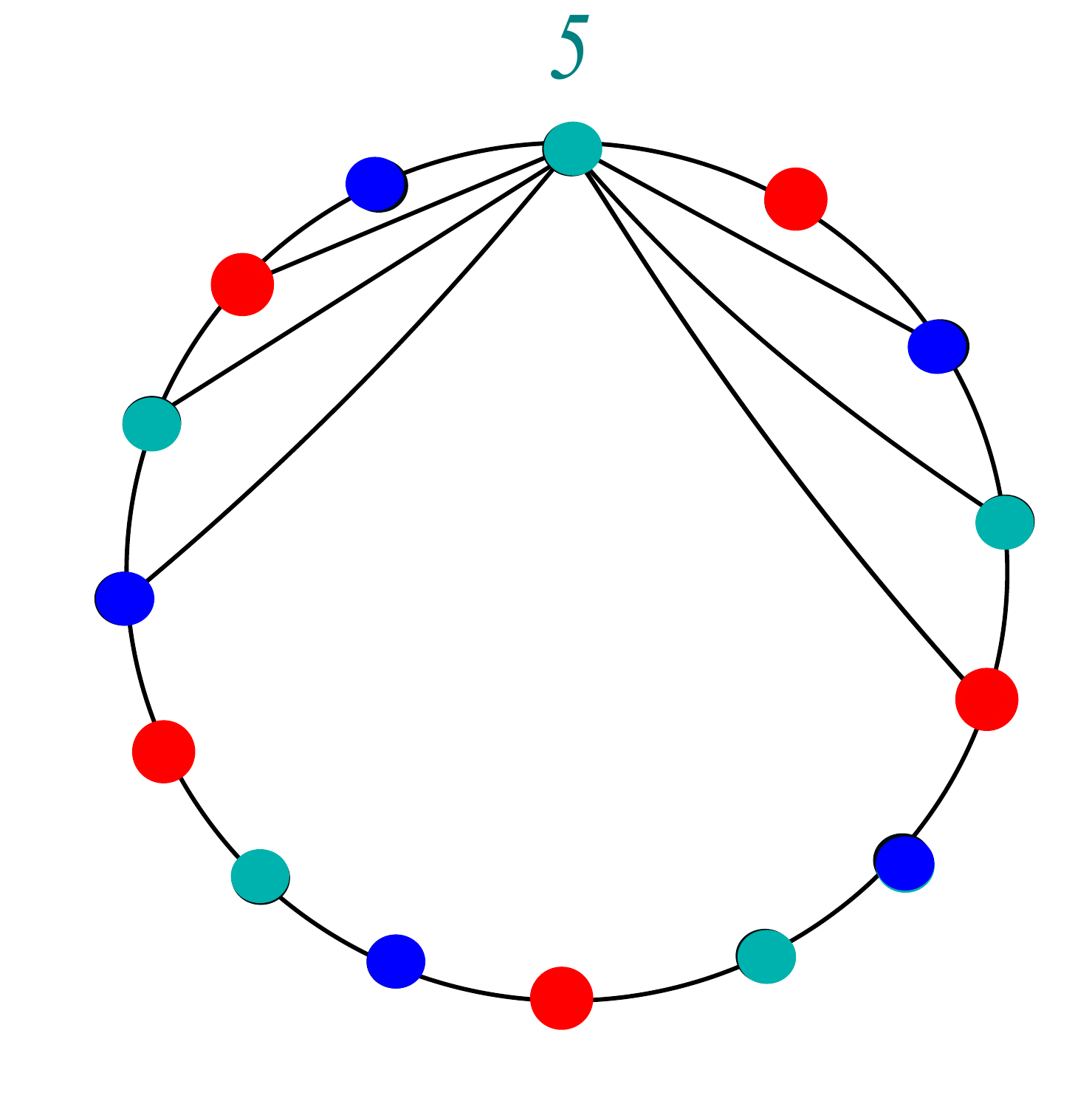}
\caption{$N[5]$ has $l_2=3$ vertices of each $S_i$, referring to the web graph $W^4_{15}$ of Example \ref{ejemplo1}.}\label{figure3}
\end{figure}
\begin{proposition}\label{kigualal2}

For every web graph $W_n^m$,
$$\gamma_{\times l_2}(W_n^m)=\frac{n}{\mu}.$$
\end{proposition}

\begin{proof}
 Lemma \ref{l_2domino} implies that each $S_i$ is an $l_2$-tuple dominating set in $W_n^m$ and thus $\gamma_{\times l_2}(W_n^m)\leq\frac{n}{\mu}$.

If $\gamma_{\times l_2}(W_n^m)<\frac{n}{\mu}$, then there exists an $l_2$-tuple dominating set $D$  of $W_n^m$  with $\left|D\right|\leq \frac{n}{\mu}-1$. Therefore

$$ \sum_{v\in V(W_n^m)}|N[v]\cap D|=\sum_{v\in D}\left| N[v]\right|=\left|D\right|(2m+1)\leq nl_2-(2m+1)<n l_2.$$ 
It follows that there exists a  vertex $v\in V(W_n^m)$ such that $|N[v]\cap D|<l_2$ leading to a contradiction since $D$ is an  $l_2$-tuple dominatig set of $W_n^m$.\end{proof}

\bigskip

Following Proposition \ref{tamanocircular2}, we consider an ordering in each set  $S_i$ starting from $i$,  in such a way that each element is obtained from the previous one in this ordering as a ``circular movement'' of $2m+1$ positions. 
This is formalized in the following definition:

\begin{definition}\label{contiguos}
Given a web graph $W_n^{m}$, $i\in \{1,\ldots \mu\}$ and  $j, q$ two vertices in $S_i$,  we say that $q$ is \emph{$1$-contiguous} to $j$ when 
$q\equiv j+2m+1$ (mod $n$).
\end{definition}

\begin{remark} Recalling that $c$ indicates the quotient in  the integer division of $n$ by $2m+1$ and that $c=1$ if and only if $2m+1 > \frac{n}{2}$, when $q$ is $1$-contiguous to $j$ it holds

$$\left|N[j]\cup N[q]\right|= \left\{
\begin{array}{ccc}
n   & \; if \; & c=1 \\
2(2m+1) & \; if \; & c\geq 2.
\end{array} \right.$$

\end{remark}

Given a web graph and $i\in \{1, \ldots,\mu\}$, let us consider in $S_i$, the ordering induced by the $1$-contiguous relation introduced in  Definition \ref{contiguos}. To indicate and list the elements of $S_i$ following that ordering, we write $$\left\langle S_i\right\rangle=(s^i_1,\ldots, s^i_{\frac{n}{\mu}}).$$

\begin{example}\label{ejemplo2}
For $W_{15}^4$ (see Example \ref{ejemplo1}), we can write 
$\left\langle S_1\right\rangle=( 1, 10, 4, 13, 7 )$, $\left\langle S_2\right\rangle= (2, 11, 5, 14, 8 )$ and $\left\langle S_3\right\rangle= (3, 12, 6, 15, 9)$.
\end{example}

\begin{remark} Any subset of $c+1$ consecutive vertices from any $\left\langle S_i\right\rangle$  is a  $1$-tuple dominating set in $W_n^m$. 
\end{remark}

Given $m$, $n$, $\langle S_i \rangle$ for some $i$ and  $\alpha \leq l_2$, procedure \textbf{DOM} described below returns as output an $\alpha$-tuple dominating set in the web graph  $W^m_n$. 
In each step, it chooses from $\langle S_i \rangle$, a vertex that is  the $1$-contiguous vertex to the latest vertex added to $D$ and stops when each vertex of $W^m_n$ has at least $\alpha$ vertices in its closed neighborhood belonging to $D$.
We take into account Procedure \textbf{DIV}($t$, $w$) which returns the quotient  and the remainder  from the Integer Division between two given integer numbers $t$ and $w$. Also, Procedure \textbf{PROC}($n$,$m$,$i$) ---based on Proposition \ref{tamanocircular2}--- which returns the set  $\langle S_i \rangle$.

\medskip

\begin{algorithmic}
\Procedure{\textbf{DOM}}{$n$, $m$, $\langle S_i \rangle$, $\alpha$}
		\State $t=0$
		\State $h=1$
				\State \textbf {DIV} ($ n $, $ 2m+1 $)  and print remainder $r$
		\State $\mu=gcd(2m+1,r)$
				\State $D=\emptyset$
     \While {$h\leq \alpha$}
			\State $j=t+1$
	  		\While {$s^i_j+2m<n$ and $j\leq \frac{n}{\mu}$}
      \State $D= D\cup \{s^i_j\}$ 
      \State $j=j+1$
   			  \EndWhile
    \State $D= D\cup \{ s_j^i \}$ 
			\State $h=h+1$
			\State $t=j$ 
   	\EndWhile
\EndProcedure
\end{algorithmic}

\bigskip

\begin{example}
For the web graph $W_{15}^4$, where $\left\langle S_1\right\rangle=( 1, 10, 4,13, 7)$, if we run \textbf{procedure} {\textbf{DOM}} $\left(15,4 , \langle S_1 \rangle, 3\right)$ we obtain:

\begin{center}
\begin{tabular}{|c|c|c|}
\hline
 \textbf{Iteration} $h$  & \textbf{Se}t $D$ & $D$ \textbf{ correponds to }\\
\hline
1& $D:=\left\{1,10\right\}$ & a 1-tuple dominating  set\\
\hline
 2& $D:=D\cup \left\{4,13\right\}$& a   2-tuple dominating set\\
\hline
3& $D:=D\cup \left\{7\right\}$&  a 3-tuple dominating set\\
\hline
\end{tabular}
\end{center}
\end{example}

\bigskip

Notice that, when \textbf{DOM} ends, every vertex of the input graph is dominated by at least $k$ vertices of the output $D$, and moreover, $r_k$ vertices are possibly dominated by at least $k+1$ vertices of the set $D$, where  $1\leq r_k\leq 2m$. 

We are ready to build Algorithm 1 and from it, to prove optimality. 
We assume that the web graph in the input is given with the circular ordering introduced in Definition \ref{contiguos}.  When this is not the case  and the graph is only given through its augmented adjacency matrix, such an ordering can be computed  in such a way that, starting from any vertex, each element in the ordering is obtained from the previous one in this ordering as a ``shift'' of $2m+1$ positions.

\begin{algorithm}[ht] 
\caption{Minimum $k$-tuple dominating set (fixed $k$) of $W^m_n$}
\label{algrdom2}
\begin{algorithmic}[1]
\Require $n\in \mathbb{Z}^{+}$, $m,k\in \mathbb{Z}^{+}$ with $n \geq 2m+1\geq k$.
\Ensure A minimum $k$-tuple dominating set $D$ of $W^m_n$. 
\State \textbf{DIV}$(n, 2m+1)$ and print quotient $c$ and remainder $r$
\State  $\mu:=gcd(2m+1,r )$ 
\State   \textbf{DIV}$(2m+1, \mu)$ and print quotient $l_2$
\State   \textbf{PROC}($n$,$m$,1) and print $\langle S_1 \rangle$
	    \If {$k\leq l_2$}  $D=$ \textbf{DOM}($n$, $m$, $\langle S_1 \rangle$, $k$)	
 		\Else { ($k>l_2$)}
					\State \textbf{DIV}$(k, l_2)$  and print quotient $\tilde{c}$ and remainder $\tilde{r}$
					\State $D= \textbf{DOM} (n, m, \langle S_{1} \rangle, \tilde{r}) \cup \textbf{PROC}(n,m,2)\cup\ldots \cup \textbf{PROC}(n,m,\tilde{c}+1)$
    	\EndIf
\end{algorithmic}
\end{algorithm}

\bigskip

\begin{proposition}\label{kdomset}
For a web graph $W_n^m$ and a positive integer $k$ with $k \leq 2m+1$, the set $D$ returned by Algorithm 1 is a $k$-tuple dominating set of  $W_n^m$.
\end{proposition}

\begin{proof}
Take $v\in V( W_n^m)$. 

When $k \leq l_2$, the set returned by Algorithm 1 in step 5  is a $k$-tuple dominating set $D$  since it is the output of  \textbf{DOM}($n$, $m$, $\langle S_1 \rangle$, $k$).

When $k > l_2$, we claim that the set $D$ returned by Algorithm 1 in step 8  is a  $k$-tuple dominating set of  $W_n^m$. By applying Procedure $\textbf{DIV}(k, l_2)$ which returns $\tilde{c}$ and $\tilde{r}$ and Lemma \ref{l_2domino},  we have:

$$\left|N[v]\cap D\right| = \left|N[v] \cap [ \bigcup\limits_{i=2}^{\widetilde{c}+1}S_i \cup  \textbf{DOM} (n, m, \langle S_1 \rangle, \widetilde{r})]\right|=$$

$$\left|N[v] \cap \bigcup\limits_{i=2}^{\widetilde{c}+1}S_i\right|  + \left|N[v] \cap  \textbf{DOM} (n, m, \langle S_1 \rangle, \widetilde{r})\right|= $$

$$\sum\limits_{i=2}^{\widetilde{c}+1} \left|N[v]  \cap S_i \right|  + \left|N[v] \cap  \textbf{DOM} (n, m, \langle S_1 \rangle, \widetilde{r})\right| \geq $$
$$ \tilde{c} l_2 + \tilde{r}= k,$$
and the claim follows.
\end{proof}

\begin{theorem}\label{kdominacion}

For a web graph $W_n^m$ and a positive integer $k$ with $k \leq 2m+1$, the $k$-tuple dominating set of  $W_n^m$ returned by Algorithm 1  has  size 
$$\gamma_{\times k}(W_n^{m})=\left\lceil
\frac{kn}{2m+1}\right\rceil.$$
\end{theorem}

\begin{proof}
From \cite{AEU}, we know  that $k\left\lfloor \frac{n}{2m+1}\right\rfloor\leq\gamma_{\times k}(W_n^{m})\leq k \left\lceil \frac{n}{2m+1}\right\rceil$. When  $r=0$, it holds   $\left\lfloor \frac{n}{2m+1}\right\rfloor= \left\lceil \frac{n}{2m+1}\right\rceil=\frac{n}{2m+1}$ and the result follows immediately. 

\smallskip
We can then assume $r\neq 0$.

If $k=l_2$, the result follows from Proposition \ref{kigualal2}.

If $k \neq l_2$, let us prove first that every $k$-tuple dominating set in $W_n^m$ has at least the desired size. Let $D$ be a $k$-tuple dominating set in $W_n^m$. Then in each column of $M^*(W_n^m)$,  there are at least $k$ 1-entries among the rows associated with vertices in $D$. 
Thus, the row submatrix of the rows associated with the vertices in $D$ has at least $kn$ $1$-entries. Since each row of this submatrix has exactly $2m+1$ 1-entries, it is clear that $|D| \geq \left\lceil \frac{kn}{2m+1}\right\rceil$.  
Therefore, $\gamma_{\times k}\left(W_n^{m}\right)\geq \left\lceil\frac{kn}{2m+1}\right\rceil$.

To prove the reverse inequality, let us take the $k$-tuple dominating set $D$ returned by Algorithm 1 and count its elements.

\begin{itemize}

\item[Case 1] $k < l_2$.  
On the one hand, we have that $$\sum_{v\in V(W^m_n)} |N[v]\cap D| =\sum_{v\in D}\left|N[v]\right|=(2m+1)\left|D\right|.$$


On the other hand, observe that when  iteration $k$ of procedure \textbf{DOM} ($n$, $m$, $\langle S_1 \rangle$, $k$) ends, each vertex of $W^m_n$ has at least $k$ vertices in its closed neighborhood belonging to $D$. Thus, since a  vertex is incorporated to $D$ following the 1-contiguous relation and its closed neighborhood has size $2m+1$,  we can write

\begin{center}
$(2m+1)\left|D\right|=kn+r_k\;$  with $\;1\leq r_k\leq 2m.$
\end{center}
Notice that in the previous equality, $r_k$ cannot be  equal to  0  since $\frac{kn}{2m+1} \notin \mathbb{Z}$ ($\frac{kn}{2m+1} $ is an integer if and only if $\frac{kl_1}{l_2}$ is an integer, which cannot be the case since
$l_1$ and $l_2$
are coprime, and $k < l_2$). Besides, $r_k$ cannot be  equal to $2m+1$  since Procedure  \textbf{DOM} would have stopped a step earlier.

Since  $$\frac{kn+1}{2m+1}\leq \frac{kn+r_k}{2m+1} \leq \frac{kn+2m}{2m+1},$$   it follows that  $$\left|D\right|=\frac{kn+r_k}{2m+1}=\left\lceil\frac{kn}{2m+1}\right\rceil.$$

\item[Case 2] $k > l_2$.
Let us denote $k=l_2\widetilde{c}+\widetilde{r}$ where $0\leq \widetilde{r}\leq l_2-1$. In this case $D=\bigcup\limits_{i=2}^{\widetilde{c}+1}S_i \cup \widetilde{D}$, where  $\widetilde{D}\subseteq S_{1}$ is obtained from  procedure \textbf{DOM} ($n$, $m$, $\langle S_1 \rangle$, $\widetilde{r}$).
$$\left|D\right|=\left|\bigcup_{i=2}^{\widetilde{c}+1} S_i\right|+\left\lceil \frac{\widetilde{r}n}{2m+1}\right\rceil =\sum_{i=2}^{\widetilde{c}+1} \left|S_i\right|+\left\lceil \frac{\widetilde{r}n}{2m+1}\right\rceil =
\widetilde{c}\frac{n}{\mu} +\left\lceil \frac{\widetilde{r}n}{2m+1}\right\rceil =$$

$$\left\lceil \widetilde{c} \frac{n}{\mu}+\frac{\widetilde{r}n}{2m+1}\right\rceil=\left\lceil \frac{\widetilde{c}l_2n+\widetilde{r}n}{2m+1}\right\rceil=\left\lceil \frac{\left(\widetilde{c}l_2+\widetilde{r}\right)n}{2m+1}\right\rceil=\left\lceil \frac{kn}{2m+1}\right\rceil.$$

\end{itemize}
\end{proof}

\begin{corollary}
For each  fixed $k$ and given  a web $W^m_n$ with $k\leq 2m+1$,  Algorithm $1$ runs in linear-time.

\end{corollary}

\begin{proof}
On the one hand, it is straightforward to check that \textbf{PROC}($n$,$m$,i) returns the set $\left\langle S_i\right\rangle=\left(s_1^{i},s_2^{i}, \ldots,s_{\frac{n}{\mu}}^{i}\right)$ in time $\mathcal{O}(n)$.

Algorithm 1 calls several times \textbf{PROC}($n$,$m$,i) which takes at most $(\mu-1)\frac{3}{\mu}n$ operations.

On the other hand,  {\textbf{ DOM}}{($n$, $m$, $\langle S_i \rangle$, $\alpha$)} runs in at most 
 $16(2m+1)c+2$ steps.

In all,  Algorithm $1$ runs in linear-time.
\end{proof}












\bigskip

\medskip





\section{Conclusions}

 In this work, we developed novel algorithms for $k$-tuple domination problems in two subclasses of circular-arc graphs. The time complexity was improved, from $\mathcal{O}(n^{6k+4})$ \cite{Belmonte2013}, \cite{Bui2013} to $\mathcal{O}(n^2)$ (for co-biconvex graphs) and to $\mathcal{O}(n)$ (for web graphs). 

For co-biconvex graphs, the algorithm  developed works for each $ 2\leq k\leq |U|+3$, where $U$ is the set of universal vertices of the input  graph. We remark that,  when the augmented adjacency matrix of the input graph is given in the form of Figure \ref{esquema}, our algorithm runs in  linear-time. We think that the techniques used in Sections \ref{aux} and \ref{stable}  together with the more general result in Theorem \ref{alfas} can be further developed to make a treatment of  the  problem for the remaining values of $k$.

Second, we presented a linear-time algorithm for web graphs for every admissible value of $k$. In this way we  completed the study initiated in \cite{AEU}, where the problem had been already solved for any  web graph $W^m_n$ but only for the cases $k=2$ and $k=2m$. 

On the one hand, it would be interesting to try to extend the techniques used in this work to the whole class of concave-round graphs which contains both, co-biconvex graphs and web graphs. 

On the other hand, we believe that some of the results in this work can be modified to find also more efficient algorithms for  total $k$-domination (i.e., with open neighborhoods)  without too many modifications. 


\section{Acknowledgements} This work was partially supported by grants PICT ANPCyT 0410 (2017-2020) and 1ING631 (2018-2022).

\end{document}